\pdfoutput=1 					
\documentclass[12pt]{article}
\usepackage{amsmath,amsthm,amssymb,color,enumerate}
\usepackage{graphicx}
\usepackage[UKenglish]{babel}
\usepackage{lmodern}
\usepackage{microtype}
\usepackage{mathtools}
\usepackage{esint}
\usepackage{comment}
\usepackage{hyperref}
\usepackage{authblk}
\usepackage[utf8]{inputenc}
\usepackage{enumitem}
\usepackage{pdfpages}

\mathtoolsset{showonlyrefs}		

\theoremstyle{plain}
\newtheorem{theorem}{Theorem}
\newtheorem{thm}[theorem]{Theorem}
\newtheorem{lemma}[theorem]{Lemma}
\newtheorem{corollary}[theorem]{Corollary}

\newtheorem{proposition}[theorem]{Proposition}

\theoremstyle{definition}

\newtheorem{assumption}[theorem]{Assumption}

\theoremstyle{remark}

\newcommand*{\R}{\mathbb{R}}

\newcommand*{\N}{\mathbb{N}}

\newcommand*{\eps}{\varepsilon}

\newcommand{\sulut}[1]{\left( #1 \right)}
\newcommand{\parens}[1]{\left( #1 \right)}
\newcommand{\joukko}[1]{\left\{ #1 \right\}}
\newcommand{\abs}[1]{\left\lvert #1 \right\rvert}

\newcommand{\norm}[1]{\left\| #1 \right\|}
\newcommand{\der}{\thinspace\mathrm{d}}

\DeclareMathOperator{\dive}{div}

\DeclareMathOperator*{\esssup}{ess \, sup}
\DeclareMathOperator*{\essinf}{ess \, inf}

\makeatletter
\newcommand*{\mint}[1]{%
  \mint@l{#1}{}%
}
\newcommand*{\mint@l}[2]{%
  \@ifnextchar\limits{%
    \mint@l{#1}%
  }{%
    \@ifnextchar\nolimits{%
      \mint@l{#1}%
    }{%
      \@ifnextchar\displaylimits{%
        \mint@l{#1}%
      }{%
        \mint@s{#2}{#1}%
      }%
    }%
  }%
}
\newcommand*{\mint@s}[2]{%
  \@ifnextchar_{%
    \mint@sub{#1}{#2}%
  }{%
    \@ifnextchar^{%
      \mint@sup{#1}{#2}%
    }{%
      \mint@{#1}{#2}{}{}%
    }%
  }%
}
\def\mint@sub#1#2_#3{%
  \@ifnextchar^{%
    \mint@sub@sup{#1}{#2}{#3}%
  }{%
    \mint@{#1}{#2}{#3}{}%
  }%
}
\def\mint@sup#1#2^#3{%
  \@ifnextchar_{%
    \mint@sup@sub{#1}{#2}{#3}%
  }{%
    \mint@{#1}{#2}{}{#3}%
  }%
}
\def\mint@sub@sup#1#2#3^#4{%
  \mint@{#1}{#2}{#3}{#4}%
}
\def\mint@sup@sub#1#2#3_#4{%
  \mint@{#1}{#2}{#4}{#3}%
}
\newcommand*{\mint@}[4]{%
  \mathop{}%
  \mkern-\thinmuskip
  \mathchoice{%
    \mint@@{#1}{#2}{#3}{#4}%
        \displaystyle\textstyle\scriptstyle
  }{%
    \mint@@{#1}{#2}{#3}{#4}%
        \textstyle\scriptstyle\scriptstyle
  }{%
    \mint@@{#1}{#2}{#3}{#4}%
        \scriptstyle\scriptscriptstyle\scriptscriptstyle
  }{%
    \mint@@{#1}{#2}{#3}{#4}%
        \scriptscriptstyle\scriptscriptstyle\scriptscriptstyle
  }%
  \mkern-\thinmuskip
  \int#1%
  \ifx\\#3\\\else_{#3}\fi
  \ifx\\#4\\\else^{#4}\fi  
}
\newcommand*{\mint@@}[7]{%
  \begingroup
    \sbox0{$#5\int\m@th$}%
    \sbox2{$#5\int_{}\m@th$}%
    \dimen2=\wd0 %
    \let\mint@limits=#1\relax
    \ifx\mint@limits\relax
      \sbox4{$#5\int_{\kern1sp}^{\kern1sp}\m@th$}%
      \ifdim\wd4>\wd2 %
        \let\mint@limits=\nolimits
      \else
        \let\mint@limits=\limits
      \fi
    \fi
    \ifx\mint@limits\displaylimits
      \ifx#5\displaystyle
        \let\mint@limits=\limits
      \fi
    \fi
    \ifx\mint@limits\limits
      \sbox0{$#7#3\m@th$}%
      \sbox2{$#7#4\m@th$}%
      \ifdim\wd0>\dimen2 %
        \dimen2=\wd0 %
      \fi
      \ifdim\wd2>\dimen2 %
        \dimen2=\wd2 %
      \fi
    \fi
    \rlap{%
      $#5%
        \vcenter{%
          \hbox to\dimen2{%
            \hss
            $#6{#2}\m@th$%
            \hss
          }%
        }%
      $%
    }%
  \endgroup
}

\setlist{nolistsep} 

\title{Recovering a variable exponent}

\author{Tommi Brander \\ tommi.brander@ntnu.no}
\affil{Norwegian University of Science and Technology}

\author{Jarkko Siltakoski \\ jarkko.j.m.siltakoski@student.jyu.fi}
\affil{University of Jyväskylä}

\begin{document}

\maketitle

\begin{abstract}
We consider an inverse problem of recovering the non-linearity for the one dimensional variable exponent $p(x)$-Laplace equation from the Dirichlet-to-Neumann map.
The variable exponent can be recovered up to the natural obstruction of rearrangements.
The main technique is using a Müntz-Szász theorem after reducing the problem to determining a function from its $L^p$-norms.
\end{abstract}

\paragraph{Keywords} Calderón's problem, inverse problem, variable exponent, non-standard growth, Müntz-Szász theorem, approximation by polynomials, elliptic equation, quasilinear equation

\paragraph{Mathematics subject classification}
34A55, 
41A10 
34B15, 
28A25 





\section{Introduction}

Calderón's fundamental inverse problem~\cite{Calderon:1980,Feldman:Salo:Uhlmann} asks if a weight function~$\gamma$ can be recovered from Dirichlet and Neumann measurements on the boundary of a domain, when the data come from the weighted Laplace equation
\begin{equation}
-\dive \parens{\gamma \nabla u} = 0.
\end{equation}
The weight function~$\gamma$ is considered as conductivity of electricity or heat, the Dirichlet boundary values as voltage or temperature, and the Neumann boundary values as current flux or heat flux through the boundary. The equation is derived from Ohm's law
\begin{equation}
-\gamma \nabla u = I
\end{equation}
and Kirchhoff's law
\begin{equation}
\dive I = 0,
\end{equation}
or corresponding laws for heat conduction.

The problem has been generalized to many other equations, of which we are interested in non-linear and singular or degenerate elliptic ones.
The physical motivation for these is that Ohm's law is only an approximation and many real-world systems exhibit highly non-linear IV (or current-voltage) patterns.
Power law -type patterns lead to the $p$-conductivity equation
\begin{equation}
-\dive \parens{\gamma \abs{\nabla u}^{p-2} \nabla u} = 0
\end{equation}
first introduced by Salo and Zhong~\cite{Salo:Zhong:2012} and investigated further by Salo and others others~\cite{Brander:2016:jan,Brander:2016:apr,Brander:Harrach:Kar:Salo:2018,Brander:Ilmavirta:Kar:2018,Brander:Kar:Salo:2015,Guo:Kar:Salo:2016,Hannukainen:Hyvonen:Mustonen:2019,Kar:Wang}.
In these works the exponent~$1 < p < \infty$ is assumed to be a known constant, whilst the linear factor in the conductivity~$\gamma$ is the unknown.
Cârstea and Kar~\cite{Carstea:Kar} investigate a combination of linear and power law type conductivity.

On the other hand, general, but typically non-degenerate, $A$-harmonic equations
\begin{equation}
-\dive \parens{a(x,u,\nabla u) \nabla u} = 0
\end{equation}
have also been researched.
A typical method is linearizing the equation or using Carleman estimates, hence relying on completely different techniques when compared to the present work.
Quasilinearities depending on the solution, with $a(u)$, have been investigated first by Cannon~\cite{Cannon:1967} and then many others~\cite{Pilant:Rundell:1988,Lesnic:Elliott:Ingham:1996,Sun:1996,Kugler:2003,Mierzwiczak:Kolodziej:2011,Egger:Pietschmann:Schlottbom:2014,Egger:Pietschmann:Schlottbom:2017}.
Sun and Uhlmann~\cite{Sun:Uhlmann:1997} considered non-degenerate and fairly smooth dependence on $u$ and $x$. A problem with similar dependency has also been studied by Chen, Chen and Wei~\cite{Chen:Chen:Wei:2015}.
Hervas and Sun~\cite{Hervas:Sun:2002} considered smooth coefficients $a(\nabla u)$.
Muñoz and Uhlmann~\cite{Munoz:Uhlmann:2018} considered non-degenerate elliptic~$a(u, \nabla u)$.
Lassas, Liimatainen and Salo consider general non-degenerate real-analytic conductivities~\cite{Lassas:Liimatainen:Salo:2019}, while Shankar~\cite{Shankar:2019} considers a non-degenerate and very smooth $a(x, u, \nabla u) $.
We also mention the work of Nakamura and collaborators~\cite{Kang:Nakamura:2002,Carstea:Nakamura:Vashisth:2019}.

In the present work we consider the variable exponent $p(\cdot)$-Laplace equation
\begin{equation}
-\dive \parens{\gamma(x) \abs{\nabla u}^{p(x)-2} \nabla u} = 0, \label{eq:var_plap}
\end{equation}
which can be both singular and degenerate at the same time.
Also, we make no smoothness assumptions.
The mathematical issues raised by the variable exponent in the forward problem have been covered in a monograph~\cite{Diening:Harjulehto:Hasto:Ruzicka:2011}.
One physical motivation for such problems is the conductivity of electricity in certain almost-superconductive materials, where the exponent~$p(\cdot)$ is a function of temperature~\cite{Bueno:Longo:Varela:2008,Kalok:2013}, which should not be assumed constant and might very well be unknown.
Former work on inverse problems for the equation consists of a boundary determination result with interior data~\cite{Brander:Ringholm}, which does not make essential use of the variable exponent, and of a characterization of conductivities~$\gamma$ that can be recovered when the exponent~$p(\cdot)$ is known~\cite{Brander:Winterrose:2019}.

In the present paper our aim is to recover the exponent~$p(\cdot)$ from the Dirichlet-to-Neumann map.
We use the results of Brander and Winterrose~\cite{Brander:Winterrose:2019}; some basic facts about the problem are summarized as preliminaries in section~\ref{sec:prelim}, together with an unrelated lemma.
Before the preliminaries we present the problem and our main results.
In section~\ref{sec:limits}, we investigate the behaviour of the Dirichlet-to-Neumann map as the difference of Dirichlet data goes to zero or infinity.
This gives explicit information about the maximum and minimum of $p$ with very few assumptions.
Finally, in section~\ref{sec:muntz}, we prove the main injectivity result via a Müntz-Szász type theorem.

\subsection{The problem}
As a forward model we consider the equation
\begin{equation}
\label{eq:forward}
-\parens{\gamma(x) \abs{u'(x)}^{p(x)-2} u'(x)}' = 0,
\end{equation}
with Dirichlet boundary data $u(a) = A$ and $u(b) = B$.
We assume, for convenience, that $A \leq B$, so that the absolute value can be removed from the equation, and write $m = B-A$. We often neglect to write the argument~$x$ in $\gamma(x)$ and $u'(x)$, but, for emphasis, keep it in $p(x)$.
We can solve the equation almost explicitly in terms of the intermediate function $m \mapsto K_m$ (sometimes we omit the subindex~$m$ and write simply $K$) defined by
\begin{equation}
\label{eq:K_m}
m = \int_a^b \gamma^{-1/(p(x)-1)} K_m^{1/(p(x)-1)} \der x.
\end{equation}
The Dirichlet-to-Neumann map (hereafter DN~map) can then be defined as
\begin{equation}
\label{eq:DN_map}
\Lambda_\gamma^p(m) = \int_a^b \gamma^{-1/(p(x)-1)} K_m^{p(x)/(p(x)-1)} \der x = mK_m.
\end{equation}
We provide a reference for this formulation in section~\ref{sec:prelim}.

\paragraph{Question:} Given $m \mapsto \Lambda_\gamma^p(m)$ and possibly~$\gamma$, can one recover $p$?

\subsection{Results}
Let $f \colon (X, \mu) \to [0, \infty]$ be a measurable function defined on a $\sigma$-finite measure space.
Define the \textit{distribution
function} of $f$ by
\begin{equation}
\label{eq:distribution_function}
\mu^{f}(t):= \mu\parens{\left\{ x\in X ; f(x)>t\right\} }\quad\text{for all }t\in[0, \infty].
\end{equation}
We say two functions are equimeasurable if and only if their distribution functions are equal.




\begin{corollary}
\label{cor:main}
Let $\gamma >0$ be a constant, and the exponents~$p_1$ and $p_2$ Lebesgue measurable and  bounded away from one and infinity.
Then $\Lambda_\gamma^{p_1} = \Lambda_\gamma^{p_2}$ if and only if $p_1$ and $p_2$ are equimeasurable with respect to the Lebesgue measure.
\end{corollary}

\begin{proof}
Theorem~\ref{thm:constant} tells that if the DN~maps agree, then the exponents~$p_i$ are equimeasurable.

Suppose now that $p_1$ and $p_2$ are equimeasurable and consider the definition of $K_m$ in equation~\eqref{eq:K_m}.
Since $\gamma$ is constant, the functions $x \mapsto (K/\gamma)^{1/(p_{i}(x)-1)}$ are equimeasurable by lemma~\ref{lemma:equi_bijection}, given any fixed constant~$K/\gamma \neq 1$, and in case of $K/\gamma=1$ they are equal.
But then their integrals agree by Tonelli's theorem, whence $K_m$ takes the same value for $p_1$ and $p_2$ for every~$m > 0$ (note that $m \mapsto K_m$ is injective, given any fixed $\gamma$ and $p$~\cite[lemma~7]{Brander:Winterrose:2019}).

We now consider equation~\eqref{eq:DN_map} written as
\begin{equation}
\Lambda_{\gamma}^{p_i}(m) = K_m \int_a^b \parens{K_m/\gamma}^{1/(p_i(x)-1)} \der x.
\end{equation}
Consider a fixed $m > 0$, whence $K_m$ takes the same value for $p_1$ and $p_2$.
It was already established that the integrals take the same value independent of $i$, which gives the equality of the DN~maps.
\end{proof}

A similar result should be reachable in one-dimensional multifrequency SPECT imaging~\cite{Brander:Ilmavirta:Tyni}.
There the corresponding question would be recovering the attenuation, given knowledge of the source term.

The corollary gives a uniqueness result, but we would like to have a reconstruction procedure, also.
Suppose $\gamma \equiv 1$.
By the proof of lemma~\ref{lem:lp_uniqueness} we know the inner products
\begin{equation}
   \int_0^M \mu(t) t^n \der t,
\end{equation}
where $\mu$ is the distribution function of $1/(p(x)-1)$ and $n \in \N$.
An orthonormal basis of the vector space with basis $\joukko{1, t, t^2, t^3, \ldots}$ gives an optimal way of recovering the distribution function~$\mu$.
The monograph~\cite[section~2.3]{Borwein:Erdelyi:1995} discusses orthonormal bases of polynomials.
Thereafter, one can recover a special rearrangement~$f^*$, the non-symmetric decreasing rearrangement~\cite{Bennett:Sharpley:1988,Kristiansson:2002}, of the function $f(x) = 1/(p(x)-1)$ by the formula
\begin{equation}
    f^{\ast}(x) = \inf\left\{ t \in [0,\infty];\mu^{f}(t)\leq x \right\},
\end{equation}
where we use the convention that $\inf\emptyset=\infty$. The function
$f^{\ast}$ is decreasing and continuous from the right. In particular,
$f$ and $f^{\ast}$ are \textit{equimeasurable}, i.e.\ their distribution
functions coincide.
Of course, one could choose to recover some other rearrangement instead.
Finally, the recovered exponent is
$r(x) = 1 + 1/f^*(x)$; it is increasing and continuous from the right.

The next theorem contains less information, but is computationally straightforward and assumes no a priori knowledge on $\gamma$. 
 It is proven in section~\ref{sec:limits}.
Write $p^- = \essinf_{a \le x \le b} p(x)$ and $p^+ = \esssup_{a \le x \le b} p(x)$.

\begin{theorem}
\label{thm:m_limits}
Let $\gamma \in L^\infty_+([a, b])$. 
The DN~map determines the quantities $p^+$, $p^-$, and if these are reached in sets of positive measure also the integrals
\begin{align}
&\int_{\joukko{x\in[a, b]; p(x) = p^-}} \gamma^{-1/(p^- - 1)}  \der x  \text{ and}\\
&\int_{\joukko{x\in[a, b]; p(x) = p^+}} \gamma^{-1/(p^+ - 1)}  \der x
\end{align}
in a constructive way.
\end{theorem}

This theorem has some corollaries.
For a constant and a priori known~$\gamma$, the measure of the set where $p(x) = p^\pm$ is found out, if it is known to be positive.
A non-constant known~$\gamma$ or a~priori bounds on an unknown~$\gamma$ would give an estimate for the sizes of the sets.

\section{Preliminaries}
\label{sec:prelim}

We give some results from the article of Brander and Winterrose~\cite{Brander:Winterrose:2019}, which build on known results for the variable exponent equation~\cite{Diening:Harjulehto:Hasto:Ruzicka:2011}.

The present paper assumes the following standing assumptions, which guarantee that there are no undue complications in understanding the existence and uniqueness of solutions for the forward problem.
\begin{assumption}[Standing assumptions]
We consider a one-dimensional interval of positive, but finite, length, i.e.\ $-\infty < a < b < \infty$.

There exists $\eps > 0$ with the following holding almost everywhere on the interval $[a, b]$:  $0 < \eps < \gamma(x) < 1/\eps$ and $1 + \eps < p(x) < 1/\eps$.
We write the assumption that $\gamma$ is essentially bounded and essentially bounded from below as $\gamma \in L^{\infty}_+$.
\end{assumption}
We then have:
\begin{enumerate}
\item The equation~\eqref{eq:forward} has a unique solution~$u(x)$.
\item The map $m \mapsto K_m$ defined by equation~\eqref{eq:K_m} is well-defined, strictly increasing, continuous bijection (from $\R_+$ to itself) with a continuous inverse.
\item The DN map defined by equation~\eqref{eq:DN_map} generalizes the usual DN~map in Calderón's problem and Calderón's problem for $p$-Laplace equation with constant~$p$ -- the DN maps are equal if $p(x)$ is constant (up to null sets).
\end{enumerate}
The paper~\cite{Brander:Winterrose:2019} did not use the observation $\Lambda^p_{\gamma} = m K_m$; the observation would likely simplify the arguments therein.

We also use the notation
\begin{align}
p^+ &= \esssup_{a \le x \le b} p(x) & \text{ and}\\
p^- &= \essinf_{a \le x \le b} p(x).
\end{align}
From the standing assumptions it follows that $1 < p^- \le p^+ < \infty$.

The following lemma states that if two functions are equimeasurable, and the same bijection acts on both of them, the composed functions are still equimeasurable.
It is needed when proving some of our main results.

\begin{lemma}
\label{lemma:equi_bijection}
Suppose that $f,g:X\rightarrow[0,\infty)$ are equimeasurable with
respect to the measures $\mu_{1}$ and $\mu_{2}$ in the sense that
\[
\mu_{1}(\left\{ x\in X;f(x)>t\right\} )=\mu_{2}(\left\{ x\in X;g(x)>t\right\} )\quad\text{for all }t\in[0,\infty).
\]
Suppose also that $\mu_{1}(X)=\mu_{2}(X)$. Let $h:I\rightarrow[0,\infty)$
be strictly monotonous and continuous, where $I$ is an interval that
contains the images of $f$ and $g$. Then $h\circ f$ and $h\circ g$
are also equimeasurable with respect to the measures $\mu_{1}$ and
$\mu_{2}$ (in the above sense).
\end{lemma}

\begin{proof}
Suppose first that $h$ is strictly increasing. Let $\tau \geq 0$.
If $\tau$ is in the image of $h$, there is $t$ such that $\tau=h(t)$.
Then, since $h$ is strictly increasing, we have by the equimeasurability assumption 
\begin{align*}
\mu_{1}(\left\{ x\in X;h\circ f(x)>\tau\right\} )= & \mu_{1}(\left\{ x\in X;h(f(x))>h(t)\right\} )\\
= & \mu_{1}(\left\{ x\in X;f(x)>t\right\} )\\
= & \mu_{2}(\left\{ x\in X;g(x)>t\right\} )\\
= & \mu_{2}(\left\{ x\in X;h\circ g(x)>\tau\right\} ).
\end{align*}
If $\tau$ is not in the image of $h$, then, since the image is a
connected set, we have either $\tau>h(t)$ for all $t\in I$ or $\tau<h(t)$
for all $t\in I$. Consequently we have either
\[
\mu_{1}(\left\{ x\in X;h\circ f(x)>\tau\right\} )=\mu_{1}(X)=\mu_{2}(X)=\mu_{2}(\left\{ x\in X;h\circ g(x)>\tau\right\} )
\]
or
\[
\mu_{1}(\left\{ x\in X;h\circ f(x)>\tau\right\} )=\mu_{1}(\emptyset)=\mu_{2}(\emptyset)=\mu_{2}(\left\{ x\in X;h\circ g(x)>\tau\right\} ).
\]
This completes the proof for strictly increasing $h$.

Suppose now that $h$ is strictly decreasing.
Then, by the same argument as above, for all $\tau \ge 0$,
\begin{equation}
\label{eq:ineq}
\mu_1 \parens{\left\{ x\in X ; h \circ f(x) < \tau \right\}} = \mu_2 \parens{\left\{ x\in X ; h \circ g(x) < \tau \right\}}.
\end{equation}
But now, by dominated convergence and for all $\tau \ge 0$,
\begin{equation}
\begin{split}
\mu_1 &\parens{\left\{ x\in X ; h \circ f(x) > \tau \right\}} = \mu_1 \parens{X} - \mu_1 \parens{\left\{ x\in X ; h \circ f(x) \le \tau \right\}} \\
&= \mu_1 \parens{X} - \lim_{\eps \to 0} \mu_1 \parens{\left\{ x\in X ; h \circ f(x) < \tau + \eps \right\}}.
\end{split}
\end{equation}
The same reasoning applies to $g$ with $\mu_2$, and by equation~\eqref{eq:ineq}, this finishes the proof.
\end{proof}

\section{Limits of voltage difference}
\label{sec:limits}
In this section we consider the situation where the difference of Dirichlet values~$m$ approaches zero or infinity. In this case we can learn something about the maximum or minimum of $p$, respectively.

\begin{proposition}
\label{prop:growth_est}
Suppose that $\gamma \in L^{\infty}_{+}$. Assume that $m\leq1$ is so small that $K_{m}\leq1$. Then for any $\varepsilon>0$ there is a constant $C=C(\gamma,p,b-a,\varepsilon)$
such that
\[
\frac{1}{C}m^{p^{+}-1}\leq K_{m}\leq Cm^{p^{+}-1-\varepsilon}
\]
and
\begin{equation}
\frac{1}{C}m^{p^{+}}\leq\varLambda_{\gamma}^{p}(m)\le Cm^{p^{+}-\varepsilon}.
\end{equation}
Moreover, if $p(x)$ reaches its essential supremum in a set of positive measure, then these estimates hold for $\varepsilon = 0$.
\end{proposition}

\begin{proof}

Since $\gamma$ is bounded away from zero and $m \leq 1$ is so small that $K_m \leq 1$, the definition of $K_m$ in \eqref{eq:K_m} implies
\begin{align*}
m & = \int_{a}^{b}\gamma^{-1/(p(x)-1)}(x)K_{m}^{\frac{1}{p(x)-1}}\der x\\
& \leq C(\gamma,p)\int_{a}^{b}K_{m}^{\frac{1}{p^{+}-1}}\der x\\
& \leq C(\gamma,p,b-a)K_{m}^{\frac{1}{p^{+}-1}}\der x,
\end{align*}
so that
\begin{equation}
K_{m}\geq C(\gamma,p,b-a)m^{p^{+}-1}.\label{eq:growth_est 1}
\end{equation}
For the other direction, let $\varepsilon>0$. Then, since $\gamma$
is bounded and $K_{m}$ is non-negative, we have
\begin{align*}
m & \geq C(\gamma,p)\int_{a}^{b}K_{m}^{\frac{1}{p(x)-1}}\der x\\
& \geq C(\gamma,p)\int_{\left\{ x \in [a,b] ;p(x)\geq p^{+}-\varepsilon\right\} }K_{m}^{\frac{1}{p^{+}-1-\varepsilon}}\der x\\
& = C(\gamma,p)\left|\left\{ x\in[a,b];p(x)\geq p^{+}-\varepsilon\right\} \right|K_{m}^{\frac{1}{p^{+}-1-\varepsilon}}.
\end{align*}
By the definition of $p^{+}$, the set $\left\{ x\in[a,b];p(x)\geq p^{+}-\varepsilon\right\} $
has a positive measure for all $\varepsilon>0$. Hence the above implies
that
\begin{equation}
K_{m}\leq C(\gamma,p,\varepsilon)m^{p^{+}-1-\varepsilon}.\label{eq:growth_est 2}
\end{equation}
Combining \eqref{eq:growth_est 1} and \eqref{eq:growth_est 2} we
arrive at the first estimate of the proposition.

The second estimate follows immediately from the identity $\Lambda_{\gamma}^{p}(m) = m K_m$.
To prove the final claim, simply repeat the above proof with $\varepsilon = 0$.
\end{proof}


\begin{lemma} Assume that $\gamma \in L^{\infty}_{+}$. Then
\label{lma:recover_sup_inf}
\begin{equation}
p^+ = \sup\joukko{q > 0 ; \lim_{m \to 0}m^{-q} \Lambda_\gamma^p = 0 } = \inf\joukko{ q > 0; \lim_{m \to 0} m^{-q}\Lambda_\gamma^p(m) = \infty }.
\end{equation}
\end{lemma}
\begin{proof}
If $q>p^{+}$, then $q=p^{+}+\varepsilon$ for some $\varepsilon>0$,
and so by proposition~\ref{prop:growth_est} we have
\[
m^{-q}\varLambda_{\gamma}^{p}(m)\geq Cm^{-p^{+}-\varepsilon}m^{p^{+}}=Cm^{-\varepsilon}\rightarrow\infty\text{ as } m\rightarrow0.
\]
If $q<p^{+}$, then $q=p^{+}-\varepsilon$ for some $\varepsilon>0$,
and so by Proposition \ref{prop:growth_est} we have
\[
m^{-q}\varLambda_{\gamma}^{p}(m)\leq Cm^{-p^{+}+\varepsilon}m^{p^{+}-\varepsilon/2}=Cm^{\varepsilon/2}\rightarrow0\text{ as }m\rightarrow0.\qedhere
\]
\end{proof}

We get similar results with $m$ large, but with $p^-$.

\begin{proposition}
Suppose that $\gamma \in L^{\infty}_{+}$. Assume that $m\ge1$ is so big that $K_{m}\ge1$.
Then for any $\varepsilon>0$ there is a constant $C=C(\gamma,p,b-a,\varepsilon)$
such that
\[
\frac{1}{C}m^{p^{-}-1}\leq K_{m}\leq Cm^{p^{-}-1+\varepsilon}
\]
and
\[
\frac{1}{C}m^{p^{-}}\leq\varLambda_{\gamma}^{p}(m)\le Cm^{p^{-}+\varepsilon}.
\]
Moreover, if $p(x)$ reaches its essential infimum in a set of positive measure, then these estimates hold for $\varepsilon = 0$.
\end{proposition}

\begin{proof}
As the proof of proposition~\ref{prop:growth_est},
but with $m$ large and $p$ estimated by $p^-$.
\end{proof}
The proof of the following lemma, too, is similar to previous proofs.

\begin{lemma} Assume that $\gamma \in L^{\infty}_{+}$. Then
\label{lma:recover_sup_inf_plus}
\begin{equation}
p^- = \inf\joukko{q > 0 ; \lim_{m \to \infty}m^{-q} \Lambda_\gamma^p = 0 } = \sup\joukko{ q > 0; \lim_{m \to \infty} m^{-q}\Lambda_\gamma^p(m) = \infty }.
\end{equation}
\end{lemma}

Now we can prove one of the inverse problem theorems.
\begin{proof}[Proof of theorem~\ref{thm:m_limits}]
Lemmas~\ref{lma:recover_sup_inf_plus} and~\ref{lma:recover_sup_inf} provide $p^-$ and $p^+$.

Suppose now $\joukko{x \in [a,b] ; p(x) = p^-}$ has positive measure, with the intention of taking $m \to \infty$. The other case is similar.
\begin{equation}
\begin{split}
m^{-p^-} \Lambda_\gamma^p(m) &= m^{-p^-+1}K_m \\
&= K_m\sulut{\int_a^b \gamma^{-1/(p(x)-1)} K_m^{1/(p(x)-1)} \der x}^{-p^-+1} \\
&= \sulut{\int_a^b \gamma^{-1/(p(x)-1)} K_m^{-1/(p^--1)+1/(p(x)-1)} \der x}^{-p^-+1} \\
&= \sulut{\int_{\joukko{x \in [a,b]; p(x) = p^-}} \gamma^{-1/(p^--1)} \der x}^{-p^-+1} \\
&+ \sulut{\int_{\joukko{x \in [a,b]; p(x) > p^-}} \gamma^{-1/(p(x)-1)} K_m^{-1/(p^--1)+1/(p(x)-1)} \der x}^{-p^-+1}.
\end{split}
\end{equation}
The second integral vanishes by dominated convergence as $m \to \infty$, since then also $K_m \to \infty$.
\end{proof}

\section{Proof via Müntz-Szász}
\label{sec:muntz}
We know~\cite[proposition~26]{Brander:Winterrose:2019} that, for $n \in \N \cup \joukko{0}$, the quantities
\begin{equation}    
\int_a^b \gamma^{-1/(p(x)-1)} \parens{\frac{1}{p(x)-1}}^n \der x,
\end{equation}
can be recovered constructively from the DN~map and its derivatives with respect to~$m$.
Suppose $\gamma \equiv 1$.
Then what can be recovered are essentially the $L^n$-norms $\norm{\frac{1}{p(x)-1}}_{L^n([a, b])}^n$.
We write the weighted $L^n$-space with weight~$f$ as $L^n([a, b], f(x)\der x)$, and omit the weight when $f \equiv 1$ almost everywhere.

\begin{proposition}
\label{prop:Lpnorms}
The following $L^n$-norms are determined constructively by the DN~map:
\begin{equation}
\norm{\frac{1}{p(x)-1}}_{L^n([a, b], \gamma^{-1/(p(x)-1)}\der x)}^n.\label{eq:weighted_lpnorms}
\end{equation}
If $\gamma \equiv 1$, we get instead
\begin{equation}
\norm{\frac{1}{p(x)-1}}_{L^n([a, b])}^n.
\end{equation}
\end{proposition}

In lemma~\ref{lem:lp_uniqueness}, we show that sufficiently many $L^n$-norms of a function uniquely determine its distribution function with respect to the underlying measure. Combined with the previous proposition, this allows us to recover the distribution function of $p(x)$, given a constant~$\gamma$.
If $\gamma$ is not constant, then the previous result is still true, but the distribution function will be with respect to a measure that depends on the unknown power~$p$.
This still gives a restatement of the original problem, but not a satisfactory characterization of the exponents~$p$ which give the same DN~map.

Recently Klun~\cite{Klun:2019} and Erdélyi~\cite{Erdelyi} proved that the equality of $L^n$ norms implies the equimeasurability of the functions. However, if $\gamma$ is not identically one, then we only know the weighted $L^n$ norms~\eqref{eq:weighted_lpnorms}, where the weight depends on the unknown power $p$. For this reason we need the slightly more general statement of lemma~11 with the two different weights. The proof resembles Klun's, but we have nevertheless included it for the benefit of the reader.

To prove lemma~\ref{lem:lp_uniqueness} we use the
following M\"untz-Sz\'asz theorem (see Theorem 12.4.4 in \cite[page~235]{Boas:1954}). For an introduction to the Müntz-Szász theorem we refer to the review by Almira~\cite{Almira:2007} and the monograph of Borwein and Erdélyi~\cite{Borwein:Erdelyi:1995}.
We are aware of some previous use of Müntz-Szász theorem in unrelated inverse problems~\cite{Janno:Wolfersdorf:1996,ElBadia:HaDuong:2002,Ling:Takeuchi:2009,Kulbay:Mukanova:Sebu:2017,Rundell:Zhang:2017}.


\begin{thm}
\label{thm:munz}Suppose that $0<\lambda_{1}<\lambda_{2}<\ldots$
is a sequence of real numbers such that $\sum_{n=1}^{\infty}1/\lambda_{n}=\infty$.
Assume that $h$ is Lebesgue integrable in $(0,M)$, $M>0$, and 
\[
\int_{0}^{M}t^{\lambda_{n}}h(t)\der t=0\quad\text{for all }n\in\mathbb{N}.
\]
Then $h(t)=0$ almost everywhere in $(0,M)$.
\end{thm}

\begin{lemma}
\label{lem:lp_uniqueness}
Suppose that $1<n_{1}<n_{2}<\ldots$ is a sequence
of real numbers such that $\sum_{j=1}^{\infty}\frac{1}{n_{j}-1}=\infty.$
Let $f,g: X \rightarrow[0,\infty)$ be $\mu_1$ and $\mu_2$ measurable, respectively, where $\mu_i$ are finite measures on $X$.
Suppose also that $f$ and $g$ are bounded, and
\[
\int_{X} f^{n_{j}}\der \mu_1(x)
=
\int_{X} g^{n_{j}}\der \mu_2(x) \quad\text{for all }j\in\mathbb{N} \setminus \{0\}.
\]
 Then $\mu^f_1=\mu^g_2$ in $[0,\infty]$, where $\mu_{1}^{f}$ and $\mu_{2}^{g}$ denote distribution functions as defined in \eqref{eq:distribution_function}.
\end{lemma}

\begin{proof}
Let $M:=\max(\left\Vert f\right\Vert _{L^{\infty}},\left\Vert g\right\Vert _{L^{\infty}}).$
Then by Tonelli's theorem and a change of variables we have
\begin{align*}
\int_{X}f^{n_{j}}\der \mu_1(x) 
& =\int_{0}^{\infty} \mu_1 \parens{ \left\{ x\in X ;f(x)^{n_j}>t\right\} } \der t\\
 & =\int_{0}^{\infty}\mu_1\parens{ \left\{ x\in X;f(x)>t^{\frac{1}{n_j}}\right\} } \der t\\
 & =\int_{0}^{\infty}\mu^{f}_1(t^{\frac{1}{n_j}})\der t\\
 & =n_j \int_{0}^{\infty}t^{n_j-1}\mu^{f}_1 (t)\der t \\
 & =n_j \int_{0}^{M}t^{n_j-1}\mu^{f}_1 (t)\der t
\end{align*}
and similarly for $g$. Thus
\begin{align*}
0= & \int_{X} \parens{f^{n_j}-g^{n_j} } \der x
= n_j \int_{0}^{M}t^{n_j -1}\left(\mu^{f}_1(t)-\mu^{g}_2(t)\right)\der t
\end{align*}
for all $j\in\mathbb{N}$.
Hence, denoting $h(t) := \mu^{f}_1(t) - \mu^{g}_2(t)$,
we have $h$ integrable on $(0,M)$ and
\[
\int_{0}^{M}t^{n_j -1}h(t)\der t = 0\quad\text{for all }n\in\mathbb{N}.
\]
It follows now from theorem \ref{thm:munz} that $h(t)=0$ for almost
every $t\in(0,M)$. Since $\mu_{f}$ and $\mu_{g}$ are continuous
from the right \cite[page~37]{Bennett:Sharpley:1988}, so is $h$, and consequently
$h\equiv0$ on $[0,M]$. It follows that $\mu_1^{f}=\mu_2^{g}$ in $[0,\infty)$.
\end{proof}

Define weighted measures on $[a,b]$ by
\begin{equation}
\mu_i(E) = \int _{E} \gamma^{-1/(p_i(x) - 1)} \der x \quad \text{for all } E \subset [a,b].
\end{equation}
\begin{theorem}
\label{thm:main}
Under the standing assumptions,
if $\Lambda_\gamma^{p_1}= \Lambda_\gamma^{p_2}$,
then
$\mu_1^{p_1} = \mu_2^{p_2}$.
\end{theorem}
\begin{proof}
Since $\gamma$ and $p$ are suitably bounded, $\gamma^{-1/(p_i(x) - 1)}$ are bounded from below and above by positive numbers.
Hence, $\parens{[a, b], \mu_i}$ are finite measure spaces. By proposition~\ref{prop:Lpnorms} we get
$$
\int_a^b \parens{\frac{1}{p_i(x)-1}}^n \der \mu_i(x)
$$
for all $n \in \N$ from the DN~map.
The functions $1/(p_i(x)-1)$ are bounded.
Now lemma~\ref{lem:lp_uniqueness} gives 
\begin{equation}
\mu_{1}^{1/(p_{1}-1)}(t)=\mu_{2}^{1/(p_{2}-1)}(t)\quad\text{for all }t\in[0,\infty).\label{eq:main_1}
\end{equation}
Using this with $t=0$ and noticing that $[a,b]=\left\{ x\in[a,b]:1/(p_{i}-1)>0\right\}$, we obtain
\begin{equation}
\mu_{1}([a,b])=\mu_{1}^{1/(p_{1}-1)}(0)=\mu_{2}^{1/(p_{2}-1)}(0)=\mu_{2}([a,b]).\label{eq:main_2}
\end{equation}
Let $h(t)=1+1/t$. Then $h:(0,\infty)\rightarrow(1,\infty)$ is strictly decreasing and we have $h\circ(1/(p_{i}-1))=p_{i}$. It now follows from \eqref{eq:main_1}, \eqref{eq:main_2} and Lemma \ref{lemma:equi_bijection} that 
\[
\mu_{1}^{p_{1}}(t)=\mu_1^{h\circ(1/(p_{1}-1))}(t)=\mu_2^{h\circ(1/(p_{2}-1))}(t)=\mu_{2}^{p_{2}}(t)\quad\text{for all }t\in[0,\infty).\qedhere
\]
\end{proof}

If $\gamma\equiv1$, then the previous theorem immediately yields the equimeasurability
of $p_{1}$ and $p_{2}$ in the Lebesgue measure. If $\gamma$ is some other constant, then the statement is seemingly different, but below we show that this is only apparent. This is natural as a constant $\gamma \not = 0$ plays no role in equation~\eqref{eq:forward}, though it affects the DN~map.

\begin{thm}
\label{thm:constant}
Suppose that $\gamma>0$ is a constant and that the assumptions of theorem \ref{thm:main} hold. Then $p_1$ and $p_2$ are equimeasurable with respect to the Lebesgue measure.
\end{thm}
\begin{proof}
We only consider the case $\gamma<1$ as the case $\gamma>1$ is similar, the main differences being the use of an auxiliary variable $\tilde t = t \log (\gamma^{-1})$ and defining the function~$F$, below, to include an additional term $-(b-a)$.

By lemma \ref{lemma:equi_bijection} it suffices to show the equimeasurability of the functions
$f:=1/(p_{1}-1)$ and $g:=1/(p_{2}-1)$. By \eqref{eq:main_1} we have
\[
\int_{\left\{ x\in[a,b];f(x)>t\right\} }\gamma^{-f(x)}\der x=\int_{\left\{ x\in[a,b];g(x)>t\right\} }\gamma^{-g(x)}\der x\quad\text{for all }t\in\mathbb{R}.
\]
By denoting $\tilde{f}:=f\log\gamma^{-1}$ and $\tilde{g}:=g\log\gamma^{-1}$,
this implies that
\begin{equation}
\label{eq:const_gamma 1}
\int_{\left\{ x\in[a,b];\tilde{f}(x)\leq t\right\} }e^{\tilde{f}(x)}\der x=\int_{\left\{ x\in[a,b];\tilde{g}(x)\leq t\right\} }e^{\tilde{g}(x)}\der x\quad\text{for all }t\in\mathbb{R}.
\end{equation}
We compute for $t\geq0$
\begin{align}
\label{eq:const_gamma 2}
\int_{\left\{ x\in[a,b];\tilde{f}(x)\leq t\right\} }e^{\tilde{f}(x)}\der x= & \int_{0}^{\infty}\left|\left\{ x\in[a,b];\tilde{f}(x)\leq t\text{ and }e^{\tilde{f}(x)}>s\right\} \right|\der s\nonumber \\
= & \int_{0}^{\infty}\left|\left\{ x\in[a,b];\log s<\tilde{f}(x)\leq t\right\} \right|\der s\nonumber \\
= & \int_{-\infty}^{\infty}\left|\left\{ x\in[a,b];u<\tilde{f}(x)\leq t\right\} \right|e^{u}\der u,
\end{align}
where in the last identity we did a change of variables $s=e^{u}$.
We define for all $u\in\mathbb{R}$ the function $F(u):=\left|\left\{ x\in[a,b];\tilde{f}(x)\le u\right\} \right|$.
Since $F$ is non-decreasing, continuous from the right and $F(0)=0$,
there exists an associated Stieltjes measure~\cite[chapter~6, section~8]{Knapp:2005} (still denoted by $F$)
such that
\[
F(c,d]=F(d)-F(c)\quad\text{whenever }c<d.
\]
Then for any $u\in\mathbb{R}$ we have
\begin{align*}
\left|\left\{ x\in[a,b];u<\tilde{f}(x)\leq t\right\} \right|= & \chi_{\left\{ u<t\right\} }(F(t)-F(u))=\chi_{\left\{ u<t\right\} }\int_{(u,t]}\der F(y).
\end{align*}
We combine this with~\eqref{eq:const_gamma 2} and continue the computation by using Tonelli's theorem
to obtain 
\begin{align}
\label{eq:const gamma 3}
\int_{\left\{ x\in[a,b];\tilde{f}(x)\leq t\right\} }e^{\tilde{f}(x)}\der x= & \int_{-\infty}^{\infty}\chi_{\left\{ u<t\right\} }e^{u}\int_{(u,t]}\der F(y)\der u\nonumber \\
= & \int_{-\infty}^{\infty}\int_{(u,t]}e^{u}\der F(y)\der u\nonumber \\
= & \int_{-\infty}^{\infty}\int_{\mathbb{R}}\chi_{\left\{ u<y\leq t\right\} }e^{u}\der F(y)\der u\nonumber \\
= & \int_{\mathbb{R}}\int_{-\infty}^{\infty}\chi_{\left\{ u<y\leq t\right\} }e^{u}\der u\der F(y)\nonumber \\
= & \int_{(0,t]}\int_{-\infty}^{y}e^{u}\der u\der F(y)\nonumber \\
= & \int_{(0,t]}e^{y}\der F(y).
\end{align}
Integrating by parts (see e.g.\ \cite[page~344]{Knapp:2005}), we obtain 
\begin{align*}
\int_{(0,t]}e^{y}\der F(y)=e^{t}F(t)-e^{0}F(0)- & \int_{0}^{t}e^{y}F(y)\der y.
\end{align*}
Combining this with~\eqref{eq:const gamma 3}, we have
\[
\int_{\left\{ x\in[a,b];\tilde{f}(x)\leq t\right\} }e^{\tilde{f}(x)}\der x=e^{t}F(t)-F(0)-\int_{0}^{t}e^{y}F(y)\der y\quad\text{for all }t\geq0.
\]
Of course, an analogical identity holds for $\tilde{g}$. Thus by
applying \eqref{eq:const_gamma 1} and using that $F(0)=0=G(0)$,
we obtain
\[
e^{t}(F(t)-G(t))=\int_{0}^{t}e^{y}(F(y)-G(y))\der y\quad\text{for all }t\geq0.
\]
Denoting $h(t)=e^t\left|F(t)-G(t)\right|$, this implies
\[
h(t)\leq\int_{0}^{t}h(y)\der y\quad\text{for all }t\geq0.
\]
By iterating the above inequality or applying Grönwall's lemma, it follows that $h\equiv0$, which implies the equimeasurability of $\tilde f$ and $\tilde g$, and thereby of $f$ and $g$.
\end{proof}
If $\gamma$ were not constant, we could try following the same proof, but the sets in equations~\eqref{eq:const_gamma 1} and \eqref{eq:const_gamma 2} would have additional $x$-dependencies, making the next step infeasible.

\subsection*{Acknowledgements}
T.B.\ was funded by the FRIPRO Toppforsk project `Waves and Nonlinear Phenomena'.

\bibliographystyle{plain}
\bibliography{math}


\end{document}